\documentclass[10pt,reqno]{amsart}
\usepackage{palatino}
\usepackage{amsfonts,amsmath}
\usepackage{amsthm}
\usepackage{graphicx,enumerate}
\usepackage[hmarginratio=1:1,hscale=0.75]{geometry}
\usepackage[latin1]{inputenc}
\usepackage[all]{xy}

\usepackage{amssymb,mathtools}
\usepackage{stmaryrd}
\usepackage{hyperref}
\usepackage{pdfsync}
\usepackage{color}
\usepackage[T1]{fontenc}
\usepackage{lmodern}

\newtheorem{counter}{Counter}
\newtheorem{lem}[counter]{Lemma}

\newtheorem{thm}[counter]{Theorem}

\newcommand{\R}{\mathbb{R}}
 
\renewcommand{\H}{\mathcal{H}} 

\newcommand{\lra}{\longrightarrow}

\renewcommand{\~}{\tilde}

\newcommand{\pd}{\partial}

\newcommand{\fal}{\forall}
\newcommand{\8}{\infty}

\newcommand{\vep}{\varepsilon} 
 
\newcommand{\dt}{\delta}

\newcommand{\om}{\Omega}

\renewcommand{\d}{\,\text{d}}

\DeclareMathOperator{\supp}{\textnormal{supp}}

\definecolor{mygreen}{rgb}{0.1,0.75,0.2}

\title[Ternary Ohta-Kawasaki in 3D]{Uniform bound on the number of partitions for optimal configurations of the Ohta-Kawasaki energy in 3D}
\author[X.Y. Lu and J. Wei]{Xin Yang Lu\address{Department of Mathematical Sciences, Lakehead University, Thunder Bay, ON, Canada}
		\email{xlu8@lakeheadu.ca}
	\and Jun-cheng Wei\address{Department of Mathematics, University of British Columbia, Vancouver, BC, Canada}
		\email{jcwei@math.ubc.ca}}

\begin{document}

	\maketitle
	
	\begin{abstract}
We study a 3D ternary system derived as a sharp-interface limit of the Nakazawa-Ohta density functional theory of triblock copolymers,
which
 combines an interface
energy with a long range interaction term.
Although both the binary case in 2D and 3D, and the ternary case in 2D, are quite well studied, very little is known about the ternary case in 3D. In particular, it is even unclear whether minimizers are made of finitely
many components. In this paper we provide a positive answer to this, by proving that
the number of components in a minimizer is bounded from above by some quantity depending
only on the total masses and the interaction coefficients.
One key difficulty is that the 3D structure
prevents us from uncoupling the Coulomb-like long range interaction 
from the perimeter term, hence
the actual shape of minimizers is unknown, not even for small masses. This is due to
the lack of a quantitative isoperimetric inequality with two mass constraints in 3D,
and it
 makes
the construction of competitors significantly more delicate.
	\end{abstract}

\bigskip
	


\medskip
	


	\bigskip
	
	\section{Introduction}

Energy functionals entailing a direct competition between an attractive short-range force and a repulsive Coulombic long-range force are mathematically studied intensively in recent years to understand physical problems such as Gamow's liquid drop problem and self-assembly of block copolymers.

In Gamow's liquid drop model \cite{ga}, the volume of the nucleus $\Omega \subset \mathbb{R}^3$ is fixed, that is, $|\Omega | = m$. Here $m$ is referred to as ``mass''.
The binding energy is given by
\begin{eqnarray*}
	\mathcal{E}_{\text{liquid}}(\Omega) := \text{Per}(\Omega) + \frac{1}{8\pi} \iint_{\Omega \times \Omega} \frac{1}{|x-y|} \d x \d y,
\end{eqnarray*}
where the first term is the perimeter term which is $\mathcal{H}^2(\Omega) $, the surface area of $\Omega$, and it arises because of lower nucleon density near the nucleus boundary; the second term is the Coulombic term which is introduced due to the presence of positively charged protons \cite{rci}.  

In Ohta and Kawasaki's diblock copolymer model \cite{equilibrium}, the free energy is represented by
\begin{eqnarray*}
	\mathcal{E}_{\text{diblock}} (\Omega) := \text{Per}(\Omega) + \gamma \iint_{\Omega \times \Omega} G(x,y) \; \d x \d y,
\end{eqnarray*}
where the first term is the perimeter term $\mathcal{H}^2(\Omega) $ which favors a large ball; the second term prefers splitting and models long-range interaction between monomers due to the connectivity of different subchains in copolymer molecules. Here
\begin{eqnarray*}
	G(x,y) = \frac{1}{4\pi |x-y| } + R(x,y)
\end{eqnarray*}
is the Green's function of $-  \triangle$ operator in $\mathbb{R}^3$, $R(x,y)$ is the regular part of $G(x,y)$, and $\gamma$ is the long-range interaction coefficient which is determined by the percentage of each type monomer, the number of all monomers in a chain molecule, the repulsion between unlike monomers, and the average distance between two adjacent monomers \cite{onDerivation}. During each experiment, the total mass of each type monomer is fixed. So the energy is minimized under the mass constraint $|\Omega | = m$.

The diblock copolymer model is a model in binary systems. In this paper we study a counter model in ternary systems which was introduced by Nakazawa and Ohta to study triblock copolymers \cite{microphase}. A triblock copolymer is a chain molecule consisting of three types of subchains, a subchain of type A monomers connected to a subchain of type B monomers and then connected to a subchain of type C monomers. Block copolymers can be used as a material in 
artificial organ technology and 
controlled drug delivery.

The free energy functional of triblock copolymers used here is a sharp interface model, 
 derived as the
 $\Gamma$-limit of Nakazawa and Ohta's diffuse interface model, by Ren and Wei in \cite{Ren2003TriblockCT,Ren2003triblock}
\begin{align*}
	\mathcal{E}_{\text{triblock}} (\Omega_1, \Omega_2) := \frac{1}{2} \sum_{i=0}^2 \text{Per}(\Omega_i) + \sum_{i,j=1}^2  \gamma_{ij} \iint_{\Omega_i \times \Omega_j} G(x,y) \; \d x \d y.
\end{align*}
Here $\om_0 = (\om_1\cup \om_2)^c$, the perimeter term is defined by
\begin{align*}
	\frac{1}{2} \sum_{i=0}^2 \text{Per}(\Omega_i)  = 
	\sum_{0\le i< j\le 2}\H^2(\pd \om_i\cap \pd\om_j),
\end{align*}
and the long-range interaction coefficients $\gamma_{ij} $ form a $2\times 2$ symmetric matrix.
Using a ``droplet" scaling argument, as done by Choksi and Peletier in 
\cite{ChoksiPeletier,ChoksiPeletierdiffuse}, and by Alama, Bronsard, the first author, and Wang in
\cite{Alama2019PeriodicMO},
it can be shown that the leading order of the free energy takes the form
\begin{align}
\label{Ohta-Kawasaki energy first order}
E_0(\om_1,\om_2) = \sum_k e_0(|\om_{1,k}|,|\om_{2,k}| ), \qquad \om_i = \cup_k \om_{i,k},\quad i=1,2,
\end{align}
where
\begin{align*}
e_0 &:[0,+\8) \times [0,+\8)\lra \R,\\
e_0(m_1,m_2)&:=\inf \bigg\{ \sum_{0\le i< j\le 2}\H^2(\pd \om_i\cap \pd\om_j)
\\
&\qquad
 +
\sum_{i,j=1}^{2}\frac{ \Gamma_{ij} }{4\pi} \int_{\om_i\times \om_j} \frac{1 }{|x-y| } \d x \d y
: |\om_i |=m_i,\ i=1,2
\bigg\}, 
\end{align*}
where $\Gamma_{ij}$ is a suitable scaling of $\gamma_{ij}$, which we will present in Section \ref{Setting up the problem}.
That is, $E_0$ seeks the optimal partition $\om_i = \bigcup_k \om_{i,k}$,
with each couple $(\om_{1,k},\om_{2,k})$ minimizing $e_0$.

	\section{Setting up the problem}\label{Setting up the problem}
	The aim of this section is to introduce the main energy of this paper.
	Choksi and Peletier showed in \cite[Theorem~4.2]{ChoksiPeletier} that, 
	with the domain being the unit torus $\mathbb{T}^3$,
	in the small mass volume fraction regime,
	the first order $\Gamma$-limit
	of the energies (see \cite[Equation~(1.8)]{ChoksiPeletier})
	\begin{align*}
E_\eta^{3d} (v):=\begin{cases}
\eta \int_{\mathbb{T}^3} |\nabla v|\d x +\eta \big\|  v-\frac{1}{|\mathbb{T}^3|}
\int_{\mathbb{T}^3}v\d x
 \big\|_{\H^{-1}(\mathbb{T}^3)}^2 & \text{if } v\in BV(\mathbb{T}^3;\{0,\eta^{-3}\}),\\
 +\8&\text{otherwise},
\end{cases}
	\end{align*}
	 is of the form
	\[\text{perimeter} + \text{long range interaction},\]
	i.e., 
\begin{align}
E_0^{3d} (v) :=\begin{cases}
\sum_{k=0}^{\8} e_0(m_k) & \text{if } v=\sum_{k=0}^{\8} m_k \dt_{x_k}, \ \sum_{i=0}^{\8} m_k=M=\text{total mass},\notag \\
+\8 & \text{otherwise},\notag
\end{cases}
\end{align}
with
	\[ e_0^{3d}(m)= \inf\bigg\{ \int_{\R^3} |\nabla z| \d x+ \Gamma \|z\|_{H^{-1}(\R^3)}^2:z\in BV(\R^3;\{0,1\}), \ \|z\|_{L^1(\R^3)}=M\bigg\}. \]
The $H^{-1}$ norm can be made explicit:
\[\|z\|_{H^{-1}(\R^3)}^2 = \int_{ \R^3\times\R^3 }  G(|x-y|) z(x)z(y)\d x\d y, \]
where $G$ denotes the Green's function of $\R^3$.
	That is, the minima seeks the optimal partition, in which each component minimizes
	the energy 
	$e_0^{3d}$.
	An analogous result, but for ternary systems on the two dimensional torus,
	was obtained in \cite[Theorem~3.2]{Alama2019PeriodicMO}.

	With the same arguments as in \cite{ChoksiPeletier,Alama2019PeriodicMO},
	we can show that
	again,	
with the domain being the unit torus $\mathbb{T}^3$,
in the small mass volume fraction regime,
	the first order $\Gamma$-limit of the energies
	(which are the analogue of \cite[Equation~(1.8)]{Alama2019PeriodicMO}
			 for ternary systems in 3D)
\begin{align*}
E_{ternary,\eta}^{3d} (v_{1,\eta},v_{2,\eta}) &:=\begin{cases}
\frac{\eta}{2}\sum_{i=0}^2\int_{\mathbb{T}^3} |\nabla v_{i,\eta}|\d x 
\\+\sum_{i,j=1}^2
\eta^4\gamma_{ij}\int_{\mathbb{T}^3\times\mathbb{T}^3}G_{{\mathbb{T}^3}}(|x-y|)
v_{i,\eta}(x) v_{i,\eta}(y)\d x\d y
& \text{if } v_{1,\eta},v_{2,\eta}\in BV(\mathbb{T}^3;\{0,\frac{1}{\eta^{3}}\}),\\
+\8&\text{otherwise},
\end{cases}\\
G_{{\mathbb{T}^3}} & :=\text{Green's function of } \mathbb{T}^3 \text{ with zero average},
\end{align*}
	can be again written 
	in the form
	\begin{align}
	E_{ternary,0}^{3d} (v_1,v_2)& :=\begin{cases}
	\sum_{k=0}^{\8} e_0(m_{1,k},m_{2,k} ) & \text{if } v_i=\sum_{k=0}^{\8} m_{i,k} \dt_{x_{i,k}}, \ \sum_{k=0}^{\8} m_{i,k}=M_i,\\
	+\8 & \text{otherwise},
	\end{cases} \label{definition of the energy in ternary 3d}\\
	M_i&=\text{total mass of type }i \text{ constituent},\qquad i=1,2,\notag
	\end{align}
	where
		\begin{align*}
		e_0(m_1,m_2) = \inf\bigg\{ &\sum_{0\le i< j\le 2}\H^2(\pd \om_i\cap \pd\om_j)+ \sum_{i,j=1}^2\Gamma_{ij} \|z_i\|_{H^{-1}(\R^3)}^2:\\
		&z_i\in BV(\R^3;\{0,1\}),\  \|z_i\|_{L^1(\R^3)}=m_i,\\
		& \om_i = \supp z_i, \ i=1,2,\  |\om_1\cap \om_2|=0\bigg\},\qquad
		\om_0 = (\om_1\cup \om_2)^c,
		\end{align*} 
		and $\Gamma_{ij}\eta^{-3}=\gamma_{ij}\ge 0 $ are coefficients penalizing the Coulomb interaction.
		Observe that the problem of minimizing $	E_{ternary,0}^{3d} $ is determined
		once we fix the total masses $M_i$ and the interaction coefficients
		$\Gamma_{ij}$. Each couple of sets $(\om_1,\om_2)$, with the appropriate masses
		and minimizing
		$e_0$, is referred to as a ``cluster''. 
Similarly to \cite{Alama2019PeriodicMO}, 
it is not restrictive to require $x_{i,k} \neq x_{i,j} $ whenever $k\neq j$, $i=1,2$, but
allowing for $x_{1,k} \neq x_{2,j} $ for some $k,j$. Thus, it is possible to 
impose $x_{1,k}=x_{2,k}=x_k$ for all $k$, and
allow for
some masses $m_{i,k}=0$, but still requiring $m_{1,k}+m_{2,k}>0$ for all $k$.
 This is due to the
ternary nature of the system, to account for the fact that at some $x_{k}$ we might
have both types of constituents (hence both $m_{1,k}$ and $m_{2,k}$ are positive), while at another $x_{j}$ we might have only one type of constituent (hence only one between $m_{1,j}$ and $m_{2,j}$ is positive).
%
	
	\medskip
	
	Next,
	we introduce the main energy of this paper:
	given connected sets $\om_1,\om_2$, with 
	$\mathbf{1}_{\om_i} \in BV(\R^3;\{0,1\})$ and
	$|\om_1\cap\om_2|=0$,
	define the energy
	\begin{align}
	E(\om_1,\om_2) &:= \sum_{0\le i< j\le 2}\H^2(\pd \om_i\cap \pd\om_j) 
+
	\sum_{i,j=1}^{2} \gamma_{ij} \int_{\om_i\times \om_j} |x-y|^{-1}\d x\d y,
		\label{Ohta-Kawasaki energy}
	\end{align}
	where $		\om_0 = (\om_1\cup \om_2)^c$. Here $\gamma_{ij}$ denote the interaction strengths, and are positive,
	of order $O(1)$.
Then, given disjoint unions
\[\Big(\bigsqcup_k \om_{1,k},\bigsqcup_k \om_{2,k}\Big), \]
with $\om_{i,k}$ being connected, the total energy of this configuration is defined
 by
\[ \mathcal{E}\Big(\bigsqcup_k \om_{1,k},\bigsqcup_k \om_{2,k}\Big):=
\sum_k E(\om_{1,k},\om_{2,k}).
 \]
Note that our main energy is quite similar to 
	\cite[Equation~(1.8)]{Alama2019PeriodicMO},
\cite[Equation~(1.8)]{Alama2019PeriodicMO}, and
\eqref{definition of the energy in ternary 3d}, with the main difference being in the 
Green's function in the interaction term; and also to 	$\mathcal{E}_{\text{liquid}}$, $\mathcal{E}_{\text{diblock}}$, and
	$\mathcal{E}_{\text{triblock}}$, with the main difference being that the interaction
	between different components are suppressed.
		
		\medskip
		
		In the following, when we say ``optimal configuration'', 
		unless otherwise specified,
		we mean a configuration $(\bigsqcup_k \om_{1,k},\bigsqcup_k \om_{2,k})$
		minimizing $\mathcal{E}$. 
		
		\medskip
		
 %
 In 2D, due to the fact that the Green's function is a logarithmic term, the interaction
 was simply the product of the masses, hence 
it was equivalent to minimize the perimeter, subject to two mass constraints.
It is well known that the double bubble is the unique such minimizer
 (see e.g. \cite{foisy1993standard,morgan2002standard} for the 2D case, and \cite{hutchings2002proof}
 for the 3D case, and also \cite{reichardt2003proof,dorff2010proof,milman2018gaussian,morgan2001geometric}). In 
 the ternary
 3D case, however, such simplification
 is not available, and the shape of the minimizers is unclear,
even for small masses. This is a significant hurdle, and studying the 
shape of minimizers is hindered by the lack of a quantitative isoperimetric inequality
with two mass constraints in 3D

Therefore,.
 a priori, it is even unclear whether optimal configurations have
finitely many clusters, as we cannot exclude the presence of infinitely many components
with very small masses.
 Our main result is to show that
 this is not the case:
 \begin{thm}\label{finite and uniform number of clusters}
 	There exists a computable constant $K=K(M_1,M_2,\gamma_{11},\gamma_{22})$ such that
 	any optimal configuration has at most $K$ clusters.
 \end{thm}

 \medskip
 
 {\bf Notation.} Since the position of the clusters is rarely relevant,
in this paper
 we denote by $B_m$ a ball of mass $m$.

 \section{Uniform upper bound on the number of clusters}
 
 The proof of Theorem \ref{finite and uniform number of clusters}
  will be split over several lemmas. Throughout the entire section,  
  $M_i$, $i=1,2$, will denote the total masses of type $i$ constituent,
  and
   $\gamma_{ij}$, $i,j=1,2$ will denote the interaction coefficients.
   These parameters completely determine the minimization problem for 
   $		\mathcal{E}_{\text{triblock}} $ in 3D.
    All the
   $M_i$ and $\gamma_{ij}$ will assumed to
    be given, and do not change throughout the section.
   Our proof will proceed as follows.
 \begin{enumerate}
 	\item  First, 
 	in Lemma \ref{not too many balls},
 	we bound from above the number of clusters made purely of one constituent type.
 	Such upper bound will depend only on $M_i$, $\gamma_{ii}$, $i=1,2$.
 	
 	\item Then, in Lemma \ref{largest clusters are not too small}, we show that the 
 	largest cluster's mass cannot be too small.
 	 	Such lower bound will depend only on $M_i$, $\gamma_{ii}$, $i=1,2$.
 	
 	\item Finally, in Lemmas \ref{sum of bubbles cannot be too small - different clusters}
 	and \ref{sum of bubbles cannot be too small - same clusters}
 	we show that the total mass of each cluster is bounded from below 
 	by a constant depending only on $M_i$, $\gamma_{ii}$, $i=1,2$.
 	Since there is only so much total mass (i.e., $M_1+M_2$), this allows us to infer 
 	Theorem \ref{finite and uniform number of clusters}.
 \end{enumerate}

 \begin{lem}
 	\label{not too many balls}
 	Consider an optimal configuration, made of clusters $\om_{i,k}$, $i=1,2$, $k\ge 1$.
 	Then 
 	\[ \# \{ k: |\om_{1,k}||\om_{2,k}|=0 \}  \]
 	is bounded from above by a constant depending only on $M_i$, $\gamma_{ii}$, $i=1,2$.
 \end{lem}
 
 \begin{proof}
 	It is well known (see e.g. \cite{knupfer2013isoperimetric,bonacini2014local,
 		muratov2014isoperimetric,frank2015compactness,
 		knupfer2016low,knupfer2019emergence}, and references therein) that there exist $m_{i,B}=m_{i,B}(\gamma_{ii})$, $i=1,2$, such that,
 	for all $m\le m_{i,B}(\gamma_{ii})$,
 	the minimizer of
 	\[\inf_{ |X| = m} \bigg\{ 
 	\H^2(\pd X ) + \gamma_{ii}\int_{X\times X}
 	|x-y|^{-1}\d x\d y \bigg\}\]
 	is given by $B_m$. 
 	Since $\H^2(\pd B_m )$ (resp. $\int_{X\times X}
 	|x-y|^{-1}\d x\d y $ ) scales like $m^{2/3}$ (resp. $m^{5/3}$), the perimeter term
 	is dominating for all sufficiently small masses. Thus there exist geometric constants
 	$m_{i,S}=m_{i,S}(\gamma_{ii})\le m_{i,B}(\gamma_{ii})$ such that
 	\begin{align*}
 	\H^2(\pd B_{m_1} ) &+ \gamma_{ii}\int_{B_{m_1}\times B_{m_1}}
 	|x-y|^{-1}\d x\d y\\
 	&\qquad+\H^2(\pd B_{m_2} ) + \gamma_{ii}\int_{B_{m_2}\times B_{m_2}}
 	|x-y|^{-1}\d x\d y\\
 	&> \H^2(\pd B_{m_1+m_2} ) + \gamma_{ii}\int_{B_{m_1+m_2}\times B_{m_1+m_2}}
 	|x-y|^{-1}\d x\d y, 
 	\end{align*}
 	for all $m_1,m_2\le m_{i,S}(\gamma_{ii})$,
 	i.e. combining the two balls is energetically favorable whenever $m_1,m_2\le m_{i,S}(\gamma_{ii})$.
 	Thus we cannot have two balls of the type $i$ constituent, both
 	with masses less than $m_{i,S}(\gamma_{ii})$. Since the total mass is $M_1+M_2<+\8$,
 	the proof is complete.
 \end{proof}

 \begin{lem}
 	\label{largest clusters are not too small}
 	Consider an optimal configuration, made of clusters $\om_{i,k}$, $i=1,2$, $k\ge 1$.
 	Then
 	\[m_i^+:=\sup_k m_{i,k}  	,\qquad m_{i,k}:=|\om_{i,k}| ,\]
 	is bounded from below by
 	\[  \min\bigg\{ \frac{M_i}{2}, \bigg(
 	\frac{c_i M_i }{2\sum_{i=1}^2 [\sqrt[3]{36\pi} M_i^{2/3} + \gamma_{ii}\int_{ B_{M_i}\times B_{M_i} } 
 		|x-y|^{-1}\d x\d y]}\bigg)^3\bigg\}
.\]
 \end{lem}
 
 Note that, curiously, this lower bound is independent of $\gamma_{12}$. As it will be clear
 from the proof, this is due to the fact that an upper bound for the energy of an optimal
 configuration is given by the energy of two balls
 of masses $M_1$ and $M_2$ respectively. Such upper bound is clearly independent of $\gamma_{12}$.
 
 \begin{proof}
 	The idea is that, for very small masses, the perimeter term is sub-addictive and dominating.
 	Assume $m_i^+\le M_i/2$, as otherwise $M_i/2$ is already a lower bound.
%
 	%
 	Note that
 	\[ E(\om_{1,k},\om_{2,k}) \ge \mathcal{S}(m_{1,k},m_{2,k}) \qquad\fal k\ge 1, \]
 	where
 	\[\mathcal{S}(m_{1},m_{2})  = \text{perimeter of the standard double bubble with masses } 
 	 m_1 \text{ and } m_2,\]
 	and,  by \cite[Theorem~4.2]{Hutchings1997TheSO} (applied with $v_1=m_1$, $x=v_2=m_2$, $n=3$)
 	\[\mathcal{S}(m_1,m_2)\ge  \sum_{i=1}^2c_i m_i^{2/3},\qquad  c_1=c_2=\frac{\sqrt[3]{36\pi}}{2}.  \]
 	 Thus the total energy of our optimal configuration satisfies
 	\begin{align*}
 	\sum_{k\ge 1}E(\om_{1,k},\om_{2,k}) \ge  \sum_{i=1}^2  c_i \sum_{k\ge 1} m_{i,k}^{2/3}.
 	\end{align*}
 	By the concavity of the function $t\mapsto t^{2/3}$, the sum $\sum_{k\ge 1} m_{i,k}^{2/3}$ is minimum
 	when $m_{i,k}\in \{0,m_i^+\}$ for all $k$. Since $\sum_{k\ge 1} m_{i,k}=M_i$,
 	there are
 	at least $\lfloor \frac{M_i}{m_i^+} \rfloor$ many clusters containing type $i$ constituents,
 	thus
 	\begin{align*}
 	\sum_{k\ge 1}E(\om_{1,k},\om_{2,k}) &\ge  \sum_{i=1}^2  c_i \sum_{k\ge 1} m_{i,k}^{2/3}
 	\ge \sum_{i=1}^2  c_i \Big\lfloor \frac{M_i}{m_i^+} \Big\rfloor (m_{i}^+)^{2/3}\\
 	&\ge \sum_{i=1}^2  c_i \frac{M_i-m_i^+}{(m_{i}^+)^{1/3}}  
 	\ge \sum_{i=1}^2  \frac{c_i}{2} \frac{M_i}{(m_{i}^+)^{1/3}}  .
 	\end{align*}
 	Since our configuration was an optimal one, its energy does not exceed that of two balls,
 	which we denote by $B_{M_1}$ and $B_{M_2}$,
 	of masses $M_1$ and $M_2$, respectively. Thus the above line continues as
 	\begin{align*}
 	\sum_{i=1}^2  \frac{c_i}{2} \frac{M_i}{(m_{i}^+)^{1/3}}  
 	&\le \sum_{k\ge 1}E(\om_{1,k},\om_{2,k})\\
 	&\le \sum_{i=1}^2 \Big[\sqrt[3]{36\pi} M_i^{2/3} + \gamma_{ii}\int_{ B_{M_i}\times B_{M_i} } 
 	|x-y|^{-1}\d x\d y\Big] ,
 	\end{align*}
 	hence 
 	\[ (m_{i}^+)^{1/3}\ge \frac{c_i M_i }{2\sum_{i=1}^2 [\sqrt[3]{36\pi} M_i^{2/3} + \gamma_{ii}\int_{ B_{M_i}\times B_{M_i} } 
 		|x-y|^{-1}\d x\d y]}, \]
 	and the proof is complete.
 \end{proof}

\begin{lem}
	\label{sum of bubbles cannot be too small - different clusters}
	Consider an optimal configuration, made of clusters $\om_{i,k}$, $i=1,2$, $k\ge 1$.
	Assume $\sup_k |\om_{1,k}|$ and $\sup_k |\om_{2,k}|$ are achieved on different clusters,
	i.e., without loss of generality,
	\[|\om_{1,1}|=m_1^+=\sup_k |\om_{1,k}|,\qquad |\om_{2,2}|=n_2^+=\sup_k |\om_{2,k}|.\]
		Then
	\[ \inf_k  \sum_{i=1}^2 |\om_{i,k}| \]
	is bounded from below by a constant depending only on $M_i$, $\gamma_{ii}$, $i=1,2$.
\end{lem}

\begin{proof}
	Consider a cluster $(\om_{1,k},\om_{2,k})$, with $k\ge 3$, and let 
	$$ |\om_{2,1}|=:m_2,\qquad |\om_{1,2}|=n_1,\qquad \vep_i:=\om_{i,k}>0,\ i=1,2.$$
	Note that $m_1^+\ge n_1$, $n_2^+\ge m_2$. 
	The construction will be slightly different depending on the values
	of $\frac{m_1^+}{m_2}$, $\frac{m_2}{n_2^+}$, and $\frac{\vep_1}{\vep_2}$.
	
	\medskip
	
	{\em Case 1:} $\frac{m_1^+}{m_2}\ge \frac{\vep_1}{\vep_2}$.
	Consider the competitor constructed
	in the following way (see Figure \ref{case1 figure}).
\begin{itemize}
	\item Move mass $\vep_1$ (resp. $rm_2$, with $r:=\frac{\vep_1}{m_1^+}\le 1$) of type I (resp. type II) constituent  
	from the cluster $(\om_{1,k},\om_{2,k})$ to $(\om_{1,1},\om_{2,1})$. This is possible
	since
	we are discussing the case
$\frac{m_1^+}{m_2}\ge \frac{\vep_1}{\vep_2}$, i.e. $rm_2=\vep_1\frac{m_2}{m_1^+}\le \vep_2$.

	\item Replace $(\om_{1,k},\om_{2,k})$ and $(\om_{1,1},\om_{2,1})$
	with $B_{\vep_2 -rm_2}$ (of type II constituent) and $
	(\~\om_{1,1},\~\om_{2,1}):=
	(1+r)^{1/3}(\om_{1,1},\om_{2,1})$ respectively, while every other cluster
		remains unaltered. 
\end{itemize}

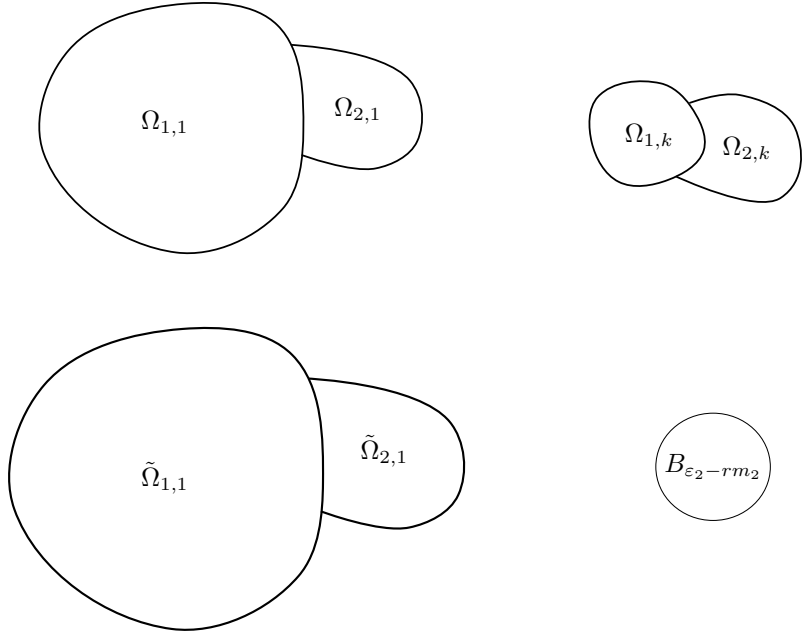
\begin{figure}
	\includegraphics[scale=.9]{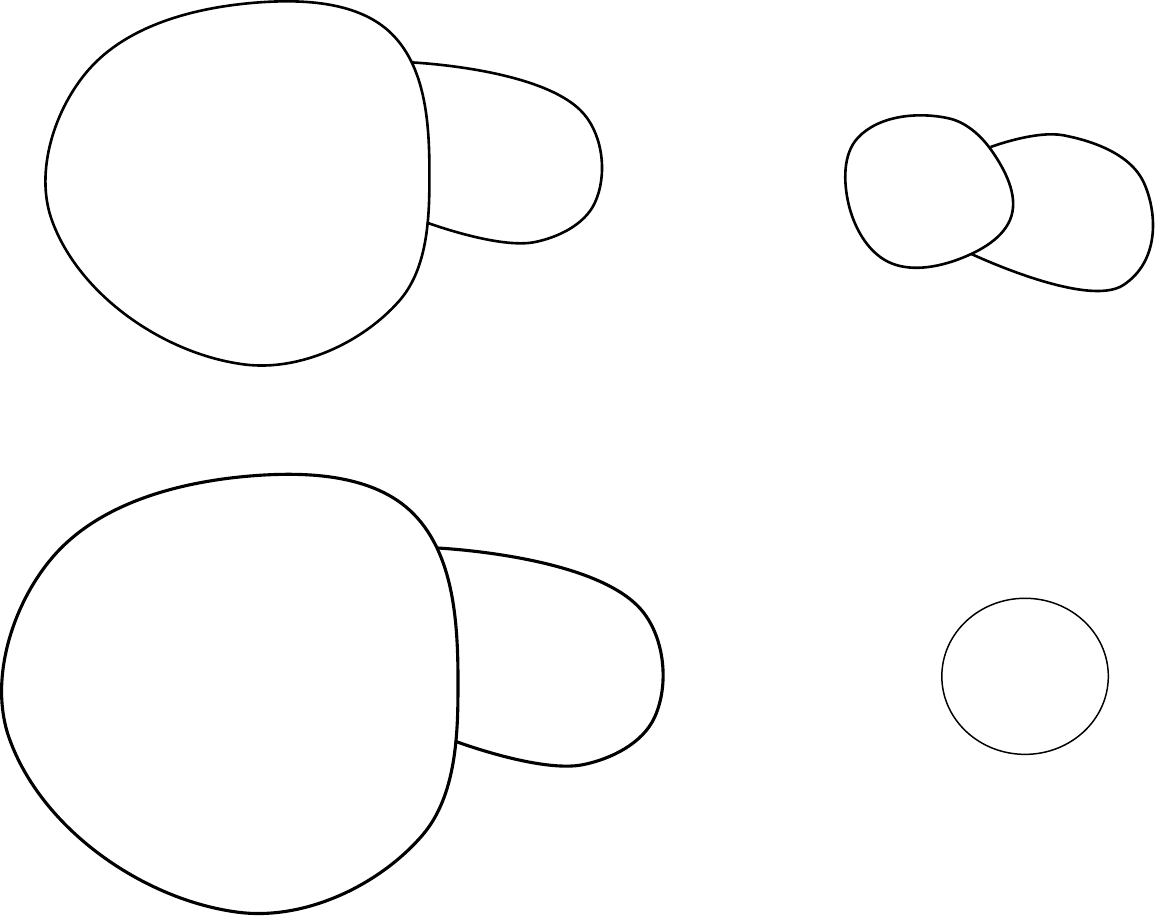}
	\put(-250,190){$\Omega_{1,1}$}
		\put(-177,195){$\Omega_{2,1}$}
	\put(-67,185){$\Omega_{1,k}$}
		\put(-30,180){$\Omega_{2,k}$}
	\put(-250,55){$\~\Omega_{1,1}$}
\put(-167,65){$\~\Omega_{2,1}$}
	\put(-52,60){$B_{\vep_2 -rm_2}$}
	\caption{Schematic representation of the construction of the competitor: 
		original clusters (top), and modified clusters (bottom).
		Though
	the objects in question are three dimensional, 
	for better clarity,
	we represented the construction in two dimensions. 
	Only the affected clusters are represented here.
The clusters 
are drawn deliberately deformed, to emphasize the fact that we do not know the clusters' precise
shapes.}
\label{case1 figure}
\end{figure}

Now we estimate the change in energy. Since our initial configuration was optimal,
\begin{align}
0&\le  E((1+r)^{1/3}(\om_{1,1},\om_{2,1}))+E(\emptyset,
B_{\vep_2 -rm_2})\notag\\
&\qquad -E(\om_{1,1},\om_{2,1})-E(\om_{1,k},\om_{2,k}).
\label{energy change - nonnegative due to optimal}
\end{align}
By a straightforward scaling argument,
\begin{align*}
E&((1+r)^{1/3}(\om_{1,1},\om_{2,1}))  \\
& = (1+r)^{2/3} \sum_{0\le i< j\le 2}\H^2(\pd \om_{i,1}\cap \pd\om_{j,1}),\qquad\om_{0,1}:=(\om_{1,1}\cup \om_{2,1})^c,\\
&\qquad +(1+r)^{5/3} \sum_{i,j=1}^2 \gamma_{ij}\int_{\om_{i,1}\times \om_{j,1}}|x-y|^{-1}\d x\d y\\
&\le (1+r)  \sum_{0\le i< j\le 2}\H^2(\pd \om_{i,1}\cap \pd\om_{j,1})+(1+3r)
\sum_{i,j=1}^2 \gamma_{ij}\int_{\om_{i,1}\times \om_{j,1}}|x-y|^{-1}\d x\d y\\
&\le (1+3r) \bigg[\underbrace{ \sum_{0\le i< j\le 2}\H^2(\pd \om_{i,1}\cap \pd\om_{j,1})+
\sum_{i,j=1}^2 \gamma_{ij}\int_{\om_{i,1}\times \om_{j,1}}|x-y|^{-1}\d x\d y}_{ = 
E(\om_{1,1},\om_{2,1}) }\bigg],
\end{align*}
where we used the estimates
\begin{align*}
(1+r)^{2/3} \le 1+r\le 1+3r,\qquad (1+r)^{5/3} \le (1+r)^2 \overset{(r\le 1)}{\le}1+3r.
\end{align*}
Thus, in view of Lemma \ref{largest clusters are not too small},
\begin{align}
E((1+r)^{1/3}(\om_{1,1},\om_{2,1}))& -E(\om_{1,1},\om_{2,1})
\le 3r E(\om_{1,1},\om_{2,1})\le \vep_1 H_1(M_1,M_2,\gamma_{11},\gamma_{22}),\notag\\
H_1(M_1,M_2,\gamma_{11},\gamma_{22})&:=
 \sum_{i=1}^2  \frac{3}{m_1^+}\bigg[\sqrt[3]{36\pi} M_i^{2/3} + \gamma_{ii}\int_{ B_{M_i}\times B_{M_i} } 
|x-y|^{-1}\d x\d y\bigg] .\label{energy change - due to scaling - case 1}
\end{align}
Now we estimate $E(\emptyset,B_{\vep_2 -rm_2}) -E(\om_{1,k},\om_{2,k})$:
\begin{align*}
E(\emptyset,B_{\vep_2 -rm_2}) &-E(\om_{1,k},\om_{2,k})  \le \mathcal{S}(0,\vep_2 -rm_2)
-\mathcal{S}(\vep_1,\vep_2 )\\
&=\mathcal{S}(0,\vep_2 -rm_2)-\mathcal{S}(\vep_1,\vep_2-rm_2 )+\mathcal{S}(\vep_1,\vep_2-rm_2 )
-\mathcal{S}(\vep_1,\vep_2 )\\
&
\le -c_1  \vep_1^{2/3} ,\qquad c_1:=\frac{\sqrt[3]{36\pi}}{2},
\end{align*}
  where the last line is due to
	\cite[Theorem~3.2]{Hutchings1997TheSO}, which gives
	\[\mathcal{S}(\vep_1,\vep_2-rm_2 )
	-\mathcal{S}(\vep_1,\vep_2 ) \le 0,\]
	 and \cite[Theorem~4.2]{Hutchings1997TheSO}
(applied with $v_1=\vep_1$, $x=v_2=\vep_2 -rm_2$, $n=3$), which gives
\begin{align*}
\mathcal{S}(\vep_1,\vep_2 -rm_2) & \ge \frac{\sqrt[3]{36\pi}}{2} [ \vep_1^{2/3} + (\vep_2 -rm_2)^{2/3}+(\vep_1+\vep_2 -rm_2)^{2/3}] \\
&\ge  \frac{\sqrt[3]{36\pi}}{2} [ \vep_1^{2/3} +2 (\vep_2 -rm_2)^{2/3}] 
=\frac{\sqrt[3]{36\pi}}{2}\vep_1^{2/3}+ \underbrace{ \sqrt[3]{36\pi}(\vep_2 -rm_2)^{2/3}}_{=\mathcal{S}(0,\vep_2 -rm_2)}.
\end{align*}
Combining with \eqref{energy change - nonnegative due to optimal} and \eqref{energy change - due to scaling - case 1} gives the necessary condition
\begin{align}
0&\le E((1+r)^{1/3}(\om_{1,1},\om_{2,1}))+E(\emptyset,
B_{\vep_2 -rm_2}) -E(\om_{1,1},\om_{2,1})-E(\om_{1,k},\om_{2,k})\notag\\
&\le \vep_1 H_1(M_1,M_2,\gamma_{11},\gamma_{22}) -c_1  \vep_1^{2/3} ,
\label{energy change - due to perimeter - case 1}
\end{align}
hence
\[\vep_1^{1/3}\ge H_1(M_1,M_2,\gamma_{11},\gamma_{22}) c_1^{-1},\]
thus completing the proof for this case.

\medskip

{\em Case 2:} $\frac{n_2^+}{n_1}\ge \frac{\vep_2}{\vep_1}$.
The competitor constructed
in a way similar to the previous case.
\begin{itemize}
	\item Move mass $\vep_2$ (resp. $rn_1$, with $r:=\frac{\vep_2}{n_2^+}\le 1$) of type II (resp. type I) constituent  
	from the cluster $(\om_{1,k},\om_{2,k})$ to $(\om_{1,2},\om_{2,2})$. This is possible
	since
	we are discussing the case
$\frac{n_2^+}{n_1}\ge \frac{\vep_2}{\vep_1}$, i.e. $rn_1=\vep_2\frac{n_1}{n_2^+}\le \vep_1$.

	\item Replace $(\om_{1,k},\om_{2,k})$ and $(\om_{1,2},\om_{2,2})$
	with $B_{\vep_1 -rn_1}$ (of type I constituent) and $(1+r)^{1/3}(\om_{1,2},\om_{2,2})$ respectively, while every other cluster
	remains unaltered. 
\end{itemize}
Then the proof proceeds like in the previous case. With the same arguments from Case 1, we obtain 
\begin{align*}
E((1+r)^{1/3}(\om_{1,2},\om_{2,2}))& -E(\om_{1,2},\om_{2,2})
\le 3r E(\om_{1,2},\om_{2,2})\le \vep_2 H_2(M_1,M_2,\gamma_{11},\gamma_{22}),\notag\\
H_2(M_1,M_2,\gamma_{11},\gamma_{22})&:=
\sum_{i=1}^2  \frac{3}{n_2^+}\bigg[\sqrt[3]{36\pi} M_i^{2/3} + \gamma_{ii}\int_{ B_{M_i}\times B_{M_i} } 
|x-y|^{-1}\d x\d y\bigg],
\end{align*}
which is the analogue of \eqref{energy change - due to scaling - case 1}, and
\begin{align*}
0&\le E((1+r)^{1/3}(\om_{1,2},\om_{2,2}))+E(\emptyset,
B_{\vep_1 -rn_1}) -E(\om_{1,2},\om_{2,2})-E(\om_{1,k},\om_{2,k})\notag\\
&\le \vep_2 H_2(M_1,M_2,\gamma_{11},\gamma_{22}) -c_2  \vep_2^{2/3},
\end{align*}
for some computable, purely geometric constant $c_2>0$,
which is the analogue of \eqref{energy change - due to perimeter - case 1}.
Thus
\[ \vep_2^{1/3} \ge H_2(M_1,M_2,\gamma_{11},\gamma_{22}) c_2^{-1}, \]
concluding the proof for this case.

\medskip

Finally, note that the above two cases are exhaustive: if Case 1
does not hold, i.e. $\frac{\vep_2}{\vep_1}< \frac{m_2}{m_1^+}$, using 
$m_1^+\ge n_1$, $n_2^+\ge m_2$, we get
\[ \frac{\vep_2}{\vep_1}< \frac{m_2}{m_1^+}\le \frac{n_2^+}{n_1}, \]
i.e. Case 2 holds.
The proof is thus complete.
\end{proof}
 
 \begin{lem}
 	\label{sum of bubbles cannot be too small - same clusters}
 	Consider an optimal configuration, made of clusters $\om_{i,k}$, $i=1,2$, $k\ge 1$.
 	Assume $\sup_k |\om_{1,k}|$ and $\sup_k |\om_{2,k}|$ are achieved on the same clusters,
 	i.e., without loss of generality,
 	\[|\om_{i,1}|=m_i^+=\sup_k |\om_{i,k}|,\quad i=1,2.\]
 	Then
 	\[ \inf_k  \sum_{i=1}^2 |\om_{i,k}| \]
 	is again bounded from below by a constant depending only on $M_i$, $\gamma_{ii}$, $i=1,2$.
 \end{lem}

 \begin{proof}
 	We rely on Lemma \ref{sum of bubbles cannot be too small - different clusters}:
Consider another cluster $(\om_{1,k},\om_{2,k})$, $k\ge 2$. Let 
$|\om_{1,k}|=\vep_1>0$, 
$|\om_{2,k}|=\vep_2>0$, and note that one of the following cases must hold.
\begin{enumerate}
	\item If $\frac{m_1^+}{m_2^+} \ge \frac{\vep_1}{\vep_2}$, then we can use the construction
	from Case 1 of Lemma \ref{sum of bubbles cannot be too small - different clusters}.
	
	\item If $\frac{m_1^+}{m_2^+} \le \frac{\vep_1}{\vep_2}$, 
	i.e. $\frac{m_2^+}{m_1^+} \ge \frac{\vep_2}{\vep_1}$,
	then we can use the construction
	from Case 2 of Lemma \ref{sum of bubbles cannot be too small - different clusters}.
\end{enumerate}
The proof is thus complete.
 \end{proof}

\section*{Acknowledgments}
XYL acknowledges the support of NSERC. The research of J. Wei is partially supported by NSERC.
We are grateful to Chong Wang for useful discussions and suggestions.

\bibliographystyle{siam}
\bibliography{bibitems}

\end{document}